\documentclass{svjour3}

\usepackage{amsfonts,amsmath}
\usepackage{graphicx}
\usepackage[utf8]{inputenc}
\usepackage{algorithm}
\usepackage{multicol}

\usepackage{hyperref}
\newcommand{\doi}[1]{DOI \href{https://doi.org/#1}{#1}} 

\usepackage[numbered]{matlab-prettifier}
\lstMakeShortInline[style=Matlab-editor]"

\DeclareMathOperator{\diag}{diag}
\DeclareMathOperator{\re}{Re}

\author{Kai Diethelm \and Safoura Hashemishahraki}

\institute{Kai Diethelm$^*$  \and
Safoura Hashemishahraki \at
Faculty of Applied Natural Sciences and Humanities (FANG), 
Technical University of Applied Sciences Würzburg-Schweinfurt, 
Ignaz-Schön-Str.~11, 97421 Schweinfurt, Germany\\
\email{$\{$kai.diethelm, safoura.hashemishahraki$\}$@thws.de}\\
$^*$corresponding author
}

\date{\today}

\title{A Stability Testing Algorithm for Incommensurate Fractional Differential Equation Systems}
\titlerunning{Stability of incommensurate fractional differential systems}

\begin{document}

\maketitle

\begin{abstract}
	We consider the question of determining whether or not a given system of fractional-order differential equations
	is (asymptotically) stable. In particular, we admit systems where each constituent equation may have its own order,
	independent of the orders of the other equations in the system, i.e.\ we discuss the so-called incommensurate 
	case. Exploiting ideas based in numerical linear algebra, we present an
	algorithm that answers this question and that is much simpler than 
	known methods. We discuss in detail the case of linear problems where the ratios of orders are rational
	and indicate how known techniques can be used to apply our findings also to general nonlinear
	problems with arbitrary orders.
	A MATLAB implementation of the code is provided.
\end{abstract}

\bigskip

\noindent\textbf{Keywords:}
	incommensurate fractional differential equation system; 
	asymptotic stability;
	numerical algorithm;
	nonlinear eigenvalue problem
	\bigskip
	
\noindent\textbf{2020 Mathematics Subject Classification:}
	65L07 $\cdot$ 34A08 $\cdot$ 34D20 

\section{Introduction}

Systems of fractional ordinary differential equations have recently become a topic of major interest,
e.g.\ in connection with the mathematical modeling of various technical or biological processes 
\cite{BM2019a,BM2019b,Di2013,T2019a,T2019b}.  Typically,
one distinguishes in this context so-called \emph{commensurate} systems where all equations of the system 
involve differential operators of the same order and \emph{incommensurate} systems where this is not the
case. A crucial question for such systems is whether or not they exhibit (asymptotic) stability. While this question
can be answered relatively easily for commensurate systems by a rather straightforward transfer of ideas used 
for differential equation systems of first order  \cite{DLL2007,M1998}, the stability analysis for incommensurate 
systems is much more
challenging. To the best of our knowledge, the only available method for finding out whether such an
incommensurate system is asymptotically stable has been described in \cite{DHTT2024}. While this algorithm
is constructive in nature, it is also technically highly involved and unwieldy to implement, and its application in 
practice is computationally very complex.  The aim of this paper is to provide an algorithm that achieves the
same goal in a much simpler manner.

To be more precise, we will throughout this paper assume that $d \in \{ 2, 3, 4, \ldots\}$ is a fixed number and that
$\alpha = (\alpha_1, \alpha_2, \ldots, \alpha_d) \in (0,1]^d$ and the constant ($d \times d$) matrix $A$ 
are given. Then we aim to investigate the $d$-dimensional differential equation system
\begin{equation}
	\label{eq:lin-system}
	D^\alpha x(t) = A x(t)
\end{equation}
with respect to the asymptotic stability of its solutions $x = (x_1, x_2, \ldots, x_d)^{\text T} : [0, \infty) \to \mathbb R^d$ 
for large arguments: Can we guarantee that
\emph{every} solution $x$ of the system \eqref{eq:lin-system} (no matter which initial values it possesses) 
satisfies $\lim_{t \to \infty} x(t) = 0$?
In eq.~\eqref{eq:lin-system} and elsewhere in this paper, we use the notation
\begin{equation}
	\label{eq:def-dalpha}
	D^\alpha x(t) = \left( D^{\alpha_1} x_1(t), D^{\alpha_2} x_2(t), \ldots, D^{\alpha_d} x_d(t) \right)^{\text T}
\end{equation}
where $D^1 y(t) = y'(t)$ denotes the classical first derivative of $y$ and, for $0 < \beta < 1$,
\begin{align}
	\label{eq:def-caputo}
	D^\beta y(t) &= D^1 J^{1-\beta} [y - y(0)] (t) \\
\intertext{and}
	J^{1-\beta} y(t) &= \frac 1 {\Gamma(1-\beta)} \int_0^t (t-\tau)^{-\beta} y(\tau) \, \text d \tau
\end{align}
are, respectively, the Caputo derivative of order $\beta$ and the Riemann-Liouville integral of order 
$1-\beta$ of the function $y$ with starting point $0$, cf.~\cite[Chapters 2 and~3]{Di2010}.

\begin{remark}
	\label{rem:incommens}
	When $\alpha_1 = \alpha_2 = \ldots = \alpha_d$, the asymptotic stability is well understood
	and can be investigated with standard methods; see, e.g., \cite[Corollary 1]{DLL2007}. We will therefore
	assume in the remainder of this paper that not all the $\alpha_k$ are identical to each other.
\end{remark}

From an analytical point of view, it is well known that the answer to the question above
is the crucial component in investigating general (possibly nonlinear) incommensurate systems 
of fractional differential equations, i.e.\ equations of the form
\begin{equation}
	\label{eq:fde}
	D^\alpha x(t) = f(t, x(t))
\end{equation}
where $f: [0, \infty) \times \mathbb R^d \to \mathbb R^d$ is such that the
system \eqref{eq:fde} possesses a unique solution $x : [0, \infty) \to \mathbb R^d$ for each
set of initial conditions $x_k(0) = x_{k,0}$ ($k = 1, 2, \ldots, d$) with arbitrary $x_{k,0} \in \mathbb R$.
Details on how the information obtained about the problem \eqref{eq:lin-system} can be 
transferred to systems of the type \eqref{eq:fde} under various conditions have already been discussed elsewhere,
e.g.\ in \cite[Section 5]{DHTT2024} and in \cite{TSSC2017}. 
Therefore, we will here concentrate on linear systems of the form \eqref{eq:lin-system} 
and only discuss the general (nonlinear) case by means of a detailed walk through an example,
see Subsection \ref{subs:nonlin-example}.

\section{Analytical Foundations}

For the analysis of systems of the type \eqref{eq:lin-system}, the concept introduced in the following definition 
is essential.

\begin{definition}
	\label{def:charfunct}
	The function 
	\[
		\chi : \mathbb C \to \mathbb C, \quad
		\lambda \mapsto \det \left( 
						\diag(\lambda^\alpha) - A \right)
	\]
	where
	\[
		\diag(\lambda^\alpha) := \begin{pmatrix}
							\lambda^{\alpha_1} & & & \\
							& \lambda^{\alpha_2} & & \\
							& & \ddots & \\
							& & & \lambda^{\alpha_d}
						\end{pmatrix}
	\]
	is called the \emph{characteristic function} of the differential equation system~\eqref{eq:lin-system}.
\end{definition}

Here and in the following, we interpret noninteger powers of a complex number in the sense of the
principal branch of the power functions, i.e.\ we say
\[
	z^\beta = |z|^\beta \exp (\text i \beta \arg(z))
\]
for all $z \in \mathbb C$ and $\beta \in \mathbb R_+ \setminus \mathbb N$.

The significance of the characteristic function in our context is apparent from the 
following result from \cite[Theorem 1]{DLL2007}.

\begin{theorem}
	\label{thm:char-zeros-1}
	Consider the differential equation system~\eqref{eq:lin-system}.
	If all complex zeros of its characteristic function $\chi$ are in the open left
	half of the complex plane then all solutions $x$ of the system satisfy $\lim_{t \to \infty} x(t) = 0$.
\end{theorem}

In other words: The system \eqref{eq:lin-system} is asymptotically stable
if all zeros $\lambda$ of its characteristic function satisfy $\re \lambda < 0$ or, equivalently, $\lambda \ne 0$ and
$|\arg (\lambda)| > \pi/2$.

The converse statement is true as well, see \cite[Theorem 1]{BK2018} where this has been shown explicitly for the
case $d=2$; it is clear that the argument can be applied for $d=3,4,\ldots$ in exactly the same way:
\begin{theorem}
	\label{thm:char-zeros-2}
	If the differential equation system~\eqref{eq:lin-system} is asymptotically
	stable then all complex zeros of its characteristic function $\chi$ are in the open left
	half of the complex plane.
\end{theorem}

In principle, these two statements reduce our problem to the task of finding all zeros of the characteristic
function. In practice, however, this is a notoriously difficult problem for which no truly efficient algorithms are
known. 

\begin{remark}
	Using arguments from complex analysis, it is easy to see that a characteristic function of a 
	differential equation system of the type \eqref{eq:lin-system} can have at most finitely many zeros,
	but it is not immediately obvious how many zeros it actually has.
\end{remark}

\section{The Basic Approach for Rational Quotients $\alpha_k / \alpha_\ell$}
\label{sec:rational}

In the initial step, we will impose a substantial restriction on the admitted values of $\alpha_1, \alpha_2, 
\ldots, \alpha_d \in (0,1]$. To formulate this condition, we set 
\[
	\tilde \alpha := \max_{k=1, 2 \ldots, d} \alpha _k.
\]
Using this notation, we shall for the moment assume that all quotients
$\alpha_k / \tilde \alpha$ are rational and defer the discussion of the remaining cases until 
Section \ref{sec:alpha-irrational}. Note that this assumption is equivalent to demanding that
all quotients $\alpha_k / \alpha_\ell$ ($k, \ell = 1, 2, \ldots, d$) be rational. Clearly, 
the condition that $\alpha_k \in \mathbb Q$ for all $k$ is sufficient but not necessary for this
assumption to be fulfilled.

To prove the asymptotic stability of the system \eqref{eq:lin-system} in this case, remember that we have to show that 
all zeros of its characteristic function $\chi$ are located in the open left half of $\mathbb C$. By definition
of $\chi$, elementary results from linear algebra imply that $\lambda \in \mathbb C$ is a zero of $\chi$ 
if and only if the linear equation system
\begin{equation}
	\label{eq:les1}
	\left( \diag(\lambda^\alpha) - A \right) v = 0
\end{equation}
possesses a nontrivial solution $v \in \mathbb R^d \setminus \{ 0 \}$.

Let us denote the $j$-th coordinate unit vector of $\mathbb R^d$ 
(in the form of a column vector) by $e_j$. Then, setting
\begin{equation}
	\label{eq:def-a}
	\alpha_0 := 0, \qquad
	A_j := e_j \cdot e_j^{\text T} \quad (j = 1, 2, \ldots, d)
	\qquad \text{ and } \qquad
	A_0 := -A,
\end{equation}
we can rewrite the system \eqref{eq:les1} in the form
\begin{equation}
	\label{eq:les2}
	\left(\sum_{j=0}^d A_j \lambda^{\alpha_j} \right) v = 0.
\end{equation}
By our assumption on the $\alpha_k$, we have for each $k = 1, 2, \ldots, d$ that
\begin{equation}
	\label{eq:alphatilde}
	\alpha_k = \tilde \alpha \cdot \frac{r_k}{s_k}
\end{equation}
with certain positive integers $r_k$ and $s_k$ where, by definition of $\tilde \alpha$,
$r_k \le s_k$ for all $k$. Without loss of generality, we may assume that $\gcd(r_k, s_k) = 1$
for all $k$. Then, setting $\sigma$ to be the least common multiple 
of $s_1, s_2, \ldots, s_d$, we obtain that
\[
	\sigma \in \mathbb N 
	\qquad \text{and} \qquad
	\sigma \ge 2
\]
(because otherwise all $s_k$ would be equal to $1$, and hence all $r_k = 1$, which would imply 
$\alpha_k = \tilde \alpha$ for all $k$, so we would be in the commensurate case that we have 
excluded from our discussion, see Remark \ref{rem:incommens})
and
\[
	\alpha_k = \tilde \alpha \cdot \frac{q_k}{\sigma}
	\quad (k = 1, 2, \ldots, d)
	\qquad \text{with } \quad
	q_k = r_k \frac \sigma {s_k} \in \{ 1, 2, \ldots, \sigma \}.
\]
We now introduce the substitution $\mu := \lambda^{\tilde \alpha / \sigma}$ in eq.~\eqref{eq:les2} and obtain
\begin{equation}
	\label{eq:les3}
	p(\mu) v = 0
	\qquad \text{where} \qquad
	p(\mu) = \sum_{j=0}^d A_j \mu^{q_j}.
\end{equation}
Here, $q_0 = 0$ and $q_j \in \{ 1, 2, \ldots, \sigma \}$ for $j=1,2,\ldots,d$, which indicates that the function $p$
is a polynomial. More precisely, if $\tilde k$ is chosen such that
$\tilde \alpha = \alpha_{\tilde k}$, then $r_{\tilde k} = s_{\tilde k} = 1$ because of \eqref{eq:alphatilde},
and therefore $q_{\tilde k} = \sigma$, so we see that the degree of the polynomial $p$
is equal to $\sigma$ and hence at least 2.

We summarize the intermediate result obtained by this construction:

\begin{theorem}
	\label{thm:reformulation}
	The following three statements are equivalent:
	\begin{enumerate}
	\item[(a)] The differential equation system \eqref{eq:lin-system} is not asymptotically stable.
	\item[(b)] The characteristic function $\chi$ of the system \eqref{eq:lin-system} has a
		zero $\lambda$ with $| \arg(\lambda) | \le \pi/2$.
	\item[(c)] There exists some $\mu \in \mathbb C$ with $| \arg (\mu) | \le \pi \tilde \alpha / (2 \sigma)$
		for which the linear equation system \eqref{eq:les3} has a nontrivial solution $v \ne 0$.
	\end{enumerate}
\end{theorem}

\begin{proof}
	The equivalence of statements (a) and (b) follows from Theorems \ref{thm:char-zeros-1} and
	\ref{thm:char-zeros-2}. 
	
	The fact that (b) and (c) are equivalent to each other is a consequence of our considerations above.
	\qed
\end{proof}

Based on Theorem \ref{thm:reformulation}, we can therefore write down the algorithm
that can be used to check the asymptotic stability of the linear differential equation system~\eqref{eq:lin-system}
in an informal way (a more formal description will be given in Algorithm \ref{alg:1} below):

\begin{enumerate}
\item Compute the matrices $A_j$ and the exponents $q_j \in \mathbb N$ ($j = 0, 1, \ldots, d$) as described above,
	thus determining the polynomial $p$.
\item Determine all $\mu \in \mathbb C$ for which the system \eqref{eq:les3} has a nontrivial solution $v \ne 0$.
\item For each $\mu$ found in Step 2, calculate $\lambda = \mu^{\sigma/\tilde \alpha}$. If any of those $\lambda$ values 
	satisfies $|\arg(\lambda)| \le \pi/2$, the system is not asymptotically stable; otherwise it is asymptotically stable.
\end{enumerate}

Step 1 of this algorithm is very easy; in particular, by eq.~\eqref{eq:def-a}, the matrices $A_j$ are
straightforward to compute: $A_0$ is just the negative of the given coefficient matrix $A$ of the system 
under consideration, and $A_j$ ($j=1,2,\ldots,d$) has a 1 in row $j$ 
and column $j$ and zeros everywhere else. Thus, these latter matrices are also extremely sparse. Step 3 is also 
completely unproblematic.

The crucial part in this algorithm is therefore Step 2. Taking into account that $p$ is a known polynomial,
we can see that---in the language of linear algebra---this step can be interpreted as solving a so-called 
\emph{nonlinear} (or, more precisely, \emph{polynomial}) \emph{eigenvalue problem} \cite{GT2017,Voss2014}.
But this is a process for which well established generally applicable methods exist, see, e.g., \cite[Section 6]{GT2017} or 
\cite{GvBMM2014}. 

For our purposes, the so-called \emph{companion form approach} \cite[Section 2]{MMMM2006} 
turns out to be particularly useful. To this end, we introduce the notation $I_m$ 
for the $(m \times m)$ identity matrix and rewrite the polynomial $p$ from eq.~\eqref{eq:les3} 
(which, as seen above, is of degree $\sigma \ge 2$) in the canonical form
\begin{equation}
	\label{eq:pol-canonical}
	p(\mu) = \sum_{j=0}^\sigma B_j \mu^j.
\end{equation}
A comparison of the coefficients between the representations for $p$ given in eqs.~\eqref{eq:les3} 
and~\eqref{eq:pol-canonical}, respectively,
yields the following relations between the $A_j$ and the $B_j$:
\begin{align*}
	B_0 &= A_0, \\
	B_j &= \sum_{k \in \{1, 2, \ldots, d : q_k = j \} } A_k	\qquad (j = 1, 2, \ldots, \sigma).
\end{align*}

\begin{remark}
	\label{rem:nofullrank}
	It is an immediate consequence of the construction of the matrices $A_j$ and $B_j$ that
	\[
		\sum_{j=1}^d A_j = I_d
		\qquad \text{ and } \qquad
		\sum_{j=1}^\sigma B_j = I_d .
	\]
	Also, neither of the matrices $A_1, A_2, \ldots, A_d$ and $B_1, B_2, \ldots, B_\sigma$ has full rank.
\end{remark}

\begin{remark}
	\label{rem:inf-ev}
	Remark \ref{rem:nofullrank} in particular states that the matrix $B_\sigma$ does not 
	have full rank. According to eq.~\eqref{eq:pol-canonical}, this matrix is the leading coefficient 
	of the polynomial $p$. Therefore, as indicated in \cite[Section 2]{MMMM2006}, the polynomial 
	eigenvalue problem \eqref{eq:les3} has at least one eigenvalue at $\infty$. 
	(For the interpretation of such eigenvalues, see \cite[Definition~2.3]{MMMM2006}.)
	However, these infinite eigenvalues will not cause any problems in our development.
\end{remark}

The next step is to construct the block matrices
\[
	X = \begin{pmatrix} 
			B_\sigma & 0 \\
			0 & I_{(\sigma-1)d}
		\end{pmatrix}
\]
and
\[
	Y = \begin{pmatrix}
			-B_{\sigma-1} & -B_{\sigma-2} & \cdots & -B_0 \\
			I_d & 0 & \cdots & 0 \\
			\vdots & \ddots & \ddots & \vdots \\
			0 & \cdots & I_d & 0
		\end{pmatrix} .
\]
It is clear that the $B_j$ are of size $d \times d$ whereas $X$ and $Y$ have the size $(\sigma d) \times (\sigma d)$.
All these matrices are very easy to compute explicitly, see Algorithm \ref{alg:1}. Moreover, both $X$ and $Y$
are very sparse. The decisive property of this approach can be expressed in the following way \cite[p.~973]{MMMM2006}:

\begin{theorem}
	\label{thm:glevp}
	Each finite eigenvalue $\mu$ of the polynomial eigenvalue problem~\eqref{eq:les3} is
	also an eigenvalue of the generalized linear eigenvalue problem $\mu X  v = Y v$
	with the same multiplicity, and vice versa.
\end{theorem}

Since $B_\sigma$ is rank deficient (see Remark \ref{rem:nofullrank}), the matrix $X$ does not have full 
rank either, and therefore, like the polynomial eigenvalue problem \eqref{eq:les3}, 
the generalized linear eigenvalue problem mentioned in Theorem \ref{thm:glevp}
also has at least one eigenvalue at $\infty$. Standard numerical methods for such eigenvalue problems
like the generalized Schur decomposition (also known as the QZ algorithm) \cite[Section 7.7]{GvL1996}
that is used in MATLAB's "eig" function \cite{MATLABeig}, can handle such cases and reliably detect
all eigenvalues, no matter whether finite or infinite, with their corresponding multiplicities.
In view of Theorem \ref{thm:glevp}, we could in principle use such an algorithm and
simply ignore the infinite eigenvalues because they are not
relevant for our stability investigations.

However, this ``brute force'' approach is not very efficient. Indeed, it is preferable to use an algorithm
that is designed to compute only the finite eigenvalues, thus avoiding the (potentially computationally
expensive) calculation of many eigenvalues that will eventually be discarded anyway.
To this end, we note that it follows from the considerations of \cite{DHTT2024} 
that the number of zeros of the characteristic polynomial $\chi$ (counting multiplicities), and hence the 
number of finite eigenvalues of the eigenvalue problem \eqref{eq:les3}, 
is equal to $N := \sum_{k=1}^d q_k$. Since $q_k \le \sigma$ for all $k$ with strict
inequality holding for at least one $k$, it is clear that $N < \sigma d$. 
(Note that, in view of the sizes of the matrices $X$ and $Y$, $\sigma d$ is the total number
of all (finite and infinite) eigenvalues of the generalized eigenvalue problem from Theorem \ref{thm:glevp}.)
Depending on the particular values of the $r_k$ and $s_k$, and hence $q_k$,
the value of $N$ may even be significantly smaller than $\sigma d$; see Table \ref{tab:results} in 
Section \ref{sec:examples} for some examples.
In any case, we can therefore use a variant of the Krylov-Schur algorithm from \cite{S2002a,S2002b}
as implemented, e.g., in MATLAB's "eigs" function \cite{MATLABeigs}, and ask this
function to compute only the finite eigenvalues, i.e.\ the $N = \sum_{k=1}^d q_k$ eigenvalues closest to zero.
This eigenvalue computation approach has the additional advantage that it is able to exploit the sparsity 
of the matrices $X$ and $Y$, thus speeding up the calculations and saving memory (because not all 
entries of the matrices $X$ and $Y$ need to be stored) even more.

Putting together all these individual items, we can now write down the complete algorithm. In the MATLAB 
or GNU/Octave programming language, it takes the form shown in Algorithm \ref{alg:1}.
The function returns a Boolean variable "stbl"
that states whether the given system is asymptotically stable. Optionally, it can in addition also return
the variable "cfzeros", an array of complex numbers with the zeros of 
the characteristic function $\chi$, and the variable "mu", another array of complex numbers 
that comprises all finite eigenvalues of the problem \eqref{eq:les3}.

\begin{algorithm}[htb]
\caption{\label{alg:1}MATLAB or GNU/Octave code for checking the asymptotic stability of 
	a homogeneous linear incommensurate fractional
	differential equation system with constant coefficients as given in \eqref{eq:lin-system}.}
\begin{lstlisting}[style=Matlab-editor]
function [stbl, cfzeros, mu] = ...
		stability_check(alphatilde, r, s, A, epsilon)
  if (nargin == 4)
    epsilon = 0.0;
  end
	
  d = length(r);
  sigma = 1;
  for s0 = s
    sigma = lcm(sigma, s0); 
  end
  q = sigma * r ./ s;

  B = zeros(d, d, sigma);
  for k = 1 : d
    B(k, k, q(k)) = -1;
  end

  X = sparse(d+1 : sigma*d, d+1 : sigma*d, 1, ...
	  		sigma*d, sigma*d, (sigma+d-1)*d);
  X(1:d, 1:d) = -sparse(B(:, :, sigma));

  Y = sparse(d+1 : sigma*d, 1 : (sigma-1)*d, 1, ...
  			sigma*d, sigma*d, (sigma+d)*d);
  for k = 1 : sigma-1
    Y(1 : d, 1 + (sigma-k-1)*d : (sigma-k)*d) ...
    			= sparse(B(:, :,k));
  end
  Y(1 : d, 1 + (sigma-1)*d : sigma*d) = sparse(A);
        
  ats = alphatilde / double(sigma);
  mu = eigs(Y, X, sum(q), 0.0);
  cfzeros = mu(abs(angle(mu)) <= pi * ats) .^ (1.0 / ats);
  stbl = (min(abs(angle(mu))) > (pi/2 + epsilon) * ats) ...
  		& (det(A) ~= 0);
end
\end{lstlisting}
\end{algorithm}

\begin{remark}
	\label{rem:categories}
	A careful inspection of our derivation above reveals that each eigenvalue $\mu$ of the 
	generalized linear eigenvalue problem $\mu X v = Y v$ falls into exactly one of the following 
	five categories:
	\begin{enumerate}
	\item $\mu = \infty$:  These eigenvalues arise because the dimension of the matrices $X$ and $Y$
		is larger than $N$. They do not correspond to any zeros of the characteristic function $\chi$
		of the fractional differential equation system and therefore, as mentioned above, 
		do not need to be computed in the first place. The call to MATLAB's "eigs"
		function takes care of this.
	\item $\mu = 0$: Because of the relation $\lambda = \mu^{\sigma / \tilde \alpha}$, the characteristic
		function $\chi$ has a zero at $\lambda = 0$ if such an eigenvalue exists. Since
		$\chi(0) = \det A$, this occurs if and only if $A$ is singular. In this case, we therefore have
		a zero of $\chi$ outside of the open left half of the complex plane, and the system is not
		asymptotically stable. The assignment in the last line of the function's program code makes 
		sure that this case is handled correctly.
	\item $0 < | \mu | < \infty$ and $\pi \tilde \alpha / \sigma < |\arg (\mu)|$: Our goal is to find the zeros 
		$\lambda \in \mathbb C$ of the characteristic function $\chi$. For each of these zeros, we have
		$| \arg(\lambda)| \le \pi$, and so the associated $\mu = \lambda^{\tilde \alpha / \sigma}$
		satisfies $|\arg(\mu)| \le \pi \tilde \alpha / \sigma$. Therefore, the eigenvalues $\mu$ with 
		$\pi \tilde \alpha / \sigma < |\arg (\mu)|$ do not correspond to any zeros of $\chi$ (i.e.,
		the number of zeros of $\chi$ is smaller than the number of eigenvalues, where we count both the
		zeros and the eigenvalues according to their respective multiplicities). Thus, such eigenvalues
		(if they exist) can also be discarded for the stability analysis. This is reflected in the
		assignment to the variable "cfzeros".
	\item $0 < | \mu | < \infty$ and $|\arg (\mu)| \le \pi \tilde \alpha / \sigma$: 
		These are the eigenvalues that are really associated to zeros of $\chi$. These eigenvalues, if they exist,
		can be further subclassified as follows:
		\begin{enumerate}
		\item $0 < | \mu | < \infty$ and $|\arg (\mu)| \le \pi \tilde \alpha / (2 \sigma)$: 
			If such an eigenvalue $\mu$ exists then the corresponding zero 
			$\lambda = \mu^{\sigma / \tilde \alpha}$ of the function $\chi$
			satisfies $| \arg(\lambda)| \le \pi/2$, i.e.\ it is located in the closed right half of the 
			complex plane, and the differential equation system is not asymptotically stable.
		\item $0 < | \mu | < \infty$ and $\pi \tilde \alpha / (2 \sigma) < |\arg (\mu)| \le \pi \tilde \alpha / \sigma$: 
			By an analog argument, this final class of eigenvalues belongs to the zeros of $\chi$ in the open left
			half of the complex plane.
		\end{enumerate}
	\end{enumerate}
	In summary, we obtain asymptotic stability of the differential equation system if and only if all finite
	eigenvalues $\mu$ fall into the one of the categories 3 or 4(b).
	
	Note also that none of the corresponding eigenvectors plays any role in our considerations.
	Therefore, it is not necessary to compute the eigenvectors.
\end{remark}

\begin{remark}
	The variable "B(:, :, k)" in the code represents  
	the matrix $-B_k$, not $B_k$.
\end{remark}

\begin{remark}
	Apart from the values $\tilde \alpha$, $r$, $s$ and $A$ introduced above,
	the function's parameter list also contains an optional parameter $\epsilon \ge 0$ with default value $0$.
	This parameter allows to introduce a ``safety margin'' into the calculation to compensate 
	for the fact that the eigenvalues of the generalized eigenvalue problem of Theorem~\ref{thm:glevp}
	can only be computed numerically and thus may be slightly inaccurate.
	The function assigns the value "true" to its return variable "stbl"
	only if $\arg(\lambda) > \pi/2 + \epsilon$ for all zeros $\lambda$ of the characteristic function $\chi$.
	In this way, the user can reduce the risk of obtaining a false positive result for the stability check
	when the system is very close to the boundary between asymptotic stability and instability.
\end{remark}

\begin{remark}
	\label{rem:algconditions}
	The code in the form provided in Algorithm \ref{alg:1} assumes, but does not verify, 
	that the parameters passed in the function call fulfil the following conditions:
	\begin{enumerate}
	\item $\tilde \alpha \in (0, 1]$,
	\item $r$ and $s$ are vectors of positive integers, both of which have the same length
		and satisfy $r_k \le s_k$ and $\gcd(r_k, s_k) = 1$ for all $k$,
	\item $\min_k \{ r_k / s_k \} < \max_k \{ r_k / s_k \} = 1$,
	\item $A$ is a square real or complex matrix with as many columns and rows as 
		the length of the vectors $r$ and $s$,
	\item $\epsilon \ge 0$.
	\end{enumerate}
	A complete version of the code that also checks whether all the actual parameters in the
	function call's parameter list indeed satisfy these conditions
	is available in the Zenodo repository \cite{DH:zenodo-stability}.
\end{remark}

\begin{remark}
	The five conditions mentioned in Remark \ref{rem:algconditions} have different backgrounds.
	Specifically:
	\begin{itemize}
	\item Condition 1 and the restriction that $r_k \le s_k$ in condition 2 together guarantee that 
		all differential equations in the given system have orders in the range $(0,1]$ and therefore that
		the system becomes well posed if each individual equation is associated with a single initial
		condition, thus avoiding any potential analytical complications.
	\item Condition 4 and the requirement that $r$ and $s$ be of the same length in condition 2 are required 
		to make sure that the problem is properly posed. If one of these assumptions is not true 
		then the given data do not match and the algorithm will crash with an error message.
	\item The left inequality in condition 3 asserts that the system is incommensurate, i.e.\ that
		not all equations have the same order. If this is not true then the algorithm still
		works, but much simpler approaches are available \cite{DLL2007}.
	\item The remaining parts of conditions 2 and 3 make sure that the dimension of the 
		linear eigenvalue problem which our algorithm constructs is not larger than necessary,
		thus keeping the computational complexity at a minimum.
	\item If condition 5 is violated then it may happen that the algorithm produces a false positive result, 
		i.e.\ it claims that the system is asymptotically stable although it actually is unstable.
	\end{itemize}	
\end{remark}

\section{Comparison with Existing Approaches}
\label{sec:comparison}

We are currently aware of only one other approach for solving the problem addressed in 
Section~\ref{sec:rational}, it being the method proposed in \cite{DHTT2024} under the 
additional restriction that $\tilde \alpha \in \mathbb Q$. 
Similar to the approach of our method presented here, the algorithm described in \cite[Section~3]{DHTT2024}
constructs a sparse matrix $B$ (not to be confused with our matrices $B_j$ above)
from the given data $A$, $r_k$ and $s_k$ ($k=1, 2, \ldots, d$)
and exploits the fact that certain powers of the eigenvalues of this matrix $B$ coincide with the zeros of the
characteristic function $\chi$, see \cite[Theorem~3.1]{DHTT2024}. 

The matrix $B$ is of size $N \times N$
where, as above, $N = \sum_{j=1}^d q_j$ with $q_j$ having the same meaning as in Section \ref{sec:rational}.
Therefore, the eigenvalue problem that needs to be solved in the approach of \cite{DHTT2024} is of a lower
dimension than the $(\sigma\cdot d)$-dimensional generalized eigenvalue problem $\mu X  v = Y v$ 
from Theorem \ref{thm:glevp} that we need to solve in the new approach described here. This can be interpreted
essentially as being due to the fact that the approach of \cite{DHTT2024} does not introduce the eigenvalues at
infinity that we implicitly generate (and later eliminate again) in our approach. 
From this perspective, the method of \cite{DHTT2024} seems to have a small advantage.

However, the construction of the matrices $X$ and $Y$ in our derivation here is extremely simple and can be achieved with 
a very low computational effort whereas, on the other hand, the construction of the matrix $B$ according to the
process described in \cite[pp.~5--6]{DHTT2024} is very involved, difficult to implement and 
requires a much larger number of arithmetic operations. 
Therefore, it is computationally expensive, and one has to expect a more substantial propagation of rounding errors.
For these reasons, we believe that the new method which we have presented in Section \ref{sec:rational}
is the more attractive alternative.

We have executed our code in MATLAB R2024b as well as in GNU Octave 8.3.0.
For all examples discussed in \cite{DHTT2024}, the results obtained there 
with the other, more complex algorithm could be reproduced in both platforms.

\section{Extension to Arbitrary $\alpha \in (0,1]^d$}
\label{sec:alpha-irrational}

If we allow the orders $\alpha_1, \alpha_2, \ldots, \alpha_d$ to be arbitrary real numbers
in the interval $(0, 1]$, i.e.\ if we give up the requirement that all quotients $\alpha_k / \alpha_\ell$
be rational, then neither the approach of our Section \ref{sec:rational} here nor the scheme from
\cite[Section 3]{DHTT2024} is generally applicable any more to find the zeros of the characteristic function,
so a different method is required.

To this end, one could for example try to address the nonlinear eigenvalue problem \eqref{eq:les1} directly with the help 
of techniques from numerical linear algebra. However, all algorithms from this field that are designed
for problems with such a structure and that we are aware of (e.g., the method from \cite{vBJM2016}) 
impose smoothness assumptions
with respect to $\lambda$ for the function on left-hand side of eq.~\eqref{eq:les1} which are not satisfied in 
our case. In particular, since $\min_{k=1, 2, \ldots, d} \alpha_k \in (0,1)$, the function $\lambda \mapsto
\diag(\lambda^\alpha) - A$ is continuous in $\mathbb C$ but does not have a first derivative at the origin.

To handle this case, we therefore
take our original differential equation system \eqref{eq:lin-system} (with arbitrary
$\alpha_k \in (0,1]$ for $k=1,2,\ldots, d$) and construct a second system
\begin{equation}
	\label{eq:lin-system2}
	D^{\alpha^*} x(t) = A x(t)
\end{equation}
with the same coefficient matrix $A$ but with rational orders 
$\alpha^* = (\alpha^*_1, \alpha^*_2, \ldots, \alpha^*_d) \in ((0,1] \cap \mathbb Q)^d$.
We denote the characteristic function of the system \eqref{eq:lin-system2} by $\chi^*$.
The idea is that, if $\alpha_k^*$ is sufficiently close to $\alpha_k$ for each $k$, then we may hope
that the zeros of $\chi^*$ are also very close to the zeros of $\chi$, and hence the stability
properties of the system \eqref{eq:lin-system2} will be the same as those of the original system~\eqref{eq:lin-system}.
The following statements justify this approach. In their formulation, we use the notation
\[
	\nu(M) = \begin{cases}
			0 & \text{ if $M$ is singular,}\\
			1 / \| M^{-1} \|_2 & \text{ if $M$ is regular,}
			\end{cases}
\]
for any $(d \times d)$ matrix $M$, where $\| \cdot \|_2$ is the usual spectral norm for $(d \times d)$ matrices,
i.e.\ the natural matrix norm induced by the Euclidean vector norm.

\begin{theorem}
	Let $d \in \mathbb N$, $\alpha \in (0,1]^d$, $\epsilon > 0$ and $A$ be a constant $(d \times d)$ matrix. 
	Then, there exist some $\rho_1, \rho_2 \in \mathbb R$ and $\alpha^* \in ((0,1] \cap \mathbb Q)^d$ 
	with the following properties:
	\begin{enumerate}
	\item $\alpha_k^* \le \alpha_k$ for all $k = 1, 2, \ldots, d$.
	\item $0  < \rho_1 < 1 \le \rho_2$.
	\item All zeros of the characteristic functions $\chi$ and $\chi^*$ are located in the annulus\linebreak
		$R(\rho_1, \rho_2) := \{ z \in \mathbb C : \rho_1 \le |z| \le \rho_2 \}$.
	\item For each $k \in \{ 1, 2, \ldots, d \}$, we have
		$\sup \{ | z^{\alpha_k} - z^{\alpha^*_k} | : z \in R(\rho_1,  \rho_2) \} < \epsilon$.
	\end{enumerate}
\end{theorem}

\begin{proof}
	This is an immediate consequence of  \cite[Proposition 4.2]{DHTT2024}.
	\qed
\end{proof}

Following \cite[Definition 4.1]{DHTT2024}, 
a vector $\alpha^*$ with the properties of this Theorem is called an $\epsilon$-rational approximation of $\alpha$
associated with the matrix $A$.
The main result in this context then reads as follows.

\begin{theorem}
	\label{thm:46}
	The system \eqref{eq:lin-system} is asymptotically stable if and only if there exist some $\epsilon > 0$
	and an $\epsilon$-rational approximation $\alpha^*$ of $\alpha$ such that 
	\[
		\min \{ \nu( \diag(z^{\alpha^*}) - A) : \re z = 0 \} \ge \epsilon .
	\]
\end{theorem}

We will give the proof in Appendix \ref{app:proof46}.

The complete algorithm for handling the problem at hand then comprises two steps:
\begin{enumerate}
\item Determine 
	\begin{enumerate}
	\item a suitable value for $\epsilon$ and
	\item an $\epsilon$-rational approximation $\alpha^*$ for $\alpha$ 
	\end{enumerate}
	such that $\min \{ \nu( \diag(z^{\alpha^*}) - A) : \re z = 0 \} \ge \epsilon$.
\item Use our Algorithm~\ref{alg:1} above to investigate the asymptotic stability of the 
	system \eqref{eq:lin-system2}. 
\end{enumerate}

For the two parts of Step 1, we may use procedures outlined in \cite{DHTT2024}. 
Specifically, for a given value of $\epsilon$, \cite[Algorithm 1]{DHTT2024} describes how to find
a corresponding $\epsilon$-rational approximation for $\alpha$ and thus addresses part (b) of Step 1. 
Moreover, at least for matrices $A$ with certain properties, \cite[Steps 1 and 2 of Algorithm 2]{DHTT2024}
provide a strategy for finding some $\epsilon$ that satisfies the requirements and therefore 
deals with part (a) of Step 1. For general matrices, however, the problem does not seem to be
completely solved yet.

\begin{remark}
	When resolving such stability questions with the help of standard computing platforms in
	finite precision arithmetic, irrational numbers can never be represented exactly. Therefore,
	the considerations of Section \ref{sec:alpha-irrational} are primarily of theoretical interest.
\end{remark}

\section{Examples}
\label{sec:examples}

To demonstrate the performance of our algorithm, we now present two examples. 

As indicated in Section \ref{sec:comparison} above, for the relatively simple 
three-or four-dimensio\-nal examples presented in \cite{DHTT2024}, our algorithm
generates the same results as the method introduced there (up to the influence of 
small rounding errors in the location of the zeros of the characteristic function).
In contrast to the approach of \cite{DHTT2024}, it is however a straightforward matter to also
consider examples in more dimensions.

\subsection{A Linear Problem}

Our first example is the problem \eqref{eq:lin-system} with the $(8 \times 8)$ matrix
\begin{equation}
	\label{eq:ex1}
	A = \left( \begin{array}{rrrrrrrr}
			26.2  &  12.5  &  14.1  &  13.4  &  18. 0 &  -24.1 &  -20.8  &  -13.4 \\
			1.1   & -1. 0  &  3.5   &  3.6   &  1.2   &       0.0  &  -4.8  &  -1. 0 \\
			25.6  &  10.5  &  10.9  &  14.1  &  16.6  &  -20.7  & -19. 0  & -13.6 \\
			-49.3  & -21.7  &  -21.5 &  -26.1 &  -27.8  &   40.7 &   33.8  &  28.8 \\
			-6.6   &  1.1  &   3.9  &   4.9  &  -6.1  &   5.5  &  -2.5  &  -0.5 \\
			-2.4   &  0.2  &  -2.6  &  -1. 0  &  3.1  &  -0.8   &  5. 0  &  6.3 \\
			-11.4  &  -6.3  &  -2.4  &  -4.5  &  -4. 0  &  11.1  &   4.6  &   8.6 \\
			32.3   & 17. 0 &   21.3  &  21.7  &  14.5 &  -28.7 &  -33.6  &  -26.7 
		\end{array} \right) .
\end{equation}
To exhibit the influence of changes in the order vector $\alpha$, we consider the
differential equation system with this matrix and the following choices of $\alpha$:
\[
\begin{array}{rrcl}
	\text{(a)} \quad & \alpha_1 &=& (0.9, 0.72, 0.54, 0.72, 0.6, 0.72, 0.18, 0.3) \smallskip \\
		&&=& \displaystyle 0.9 \cdot \left( 1, \frac 4 5, \frac 3 5, \frac 4 5, \frac 2 3, \frac 4 5, \frac 1 5, \frac 1 3 \right), 
		\smallskip \\
	\text{(b)} \quad & \alpha_2 &=& (0.96, 0.84, 0.72, 0.84, 0.72, 0.84, 0.24, 0.36) \smallskip \\
		&&= & \displaystyle 0.96 \cdot \left( 1, \frac 7 8, \frac 3 4, \frac 7 8, \frac 3 4, \frac 7 8, \frac 1 4, \frac 3 8 \right),
		\smallskip \\
	\text{(c)} \quad & \alpha_3 &=& (0.72, 0.54, 0.36, 0.54, 0.48, 0.54, 0.12, 0.18) \smallskip \\
		&&= & \displaystyle 0.72 \cdot \left( 1, \frac 3 4, \frac 1 2, \frac 3 4, \frac 2 3, \frac 3 4, \frac 1 6, \frac 1 4 \right),
		\smallskip \\
	\text{(d)} \quad & \alpha_4 &=& (0.96, 0.72, 0.84, 0.6, 0.48, 0.9, 0.12, 0.36) \smallskip \\
		&&= & \displaystyle 0.96 \cdot \left( 1, \frac 3 4, \frac 7 8, \frac 5 8, \frac 1 2, \frac {15} {16}, \frac 1 8, \frac 3 8 \right).
\end{array}
\]

In this example, we first note that the matrix $A$ has two eigenvalues in the right half of $\mathbb C$.
Therefore, the first order system $D^1 x(t) = A x(t)$ is unstable. However, it is clear that modifying the
orders of the differential operators changes the stability properties of the system, and therefore the instability
of the first order system does not imply anything about the stability or instability of the incommensurate
fractional order systems. 

Indeed, when we apply our Algorithm~\ref{alg:1} with $\alpha = \alpha_1$, 
we obtain $\sigma = 15$, so (since $d=8$) the eigenvalue problem
$\mu X v = Y v$ has dimension $\sigma d = 120$. Moreover, $\tilde \alpha = 0.9$, and we find that 
$N = \sum_{k=1}^d q_k = 78$. Hence,
the eigenvalue problem has 78 finite eigenvalues that we have plotted in Figure~\ref{fig:ex1-mu}. The blue lines
in the graph are the boundaries of the wedge $\{ z \in \mathbb C : |\arg(z) | \le \pi \tilde \alpha / \sigma \}$,
and the magenta lines bound the wedge $\{ z \in \mathbb C : |\arg(z) | \le \pi \tilde \alpha / (2\sigma) \}$. 
In the terminology of Remark \ref{rem:categories}, we see that 74 of the finite eigenvalues fall 
into category 3 and thus are irrelevant for our analysis. Moreover, no eigenvalue is in category 2. Therefore,
we only have four eigenvalues $\mu$ in category 4 that need to be inspected further to see whether
they are in subcategory 4(a) or 4(b). These four eigenvalues are located at
\begin{align*}
	\mu_{1,2} &\approx 0.9725 \pm 0.1299 \text{i} & \text{ with } & &  |\arg(\mu_{1,2})| & \approx 0.0423 \pi, \\
	\mu_{3,4} &\approx 1.0892 \pm 0.1487 \text{i} & \text{ with } & &  |\arg(\mu_{3,4})| & \approx 0.0432 \pi.
\end{align*}
For the sake of exposition, we explicitly compute the zeros $\lambda_k = \mu_k^{\sigma / \tilde \alpha}$ 
($k = 1, 2, 3, 4$) of the characteristic function $\chi$ for these eigenvalues and obtain
\begin{align*}
	\lambda_{1,2} & \approx -0.4364 \pm 0.5828 \text{i} \qquad \text{and} \\
	\lambda_{3,4} & \approx -3.0819 \pm 3.7337 \text{i},
\end{align*}
all of which have a negative real part, so the differential equation system \eqref{eq:lin-system} with
the coefficient matrix $A$ from \eqref{eq:ex1} is asymptotically stable.
The plot of the numerical solution of this system given in Fig.~\ref{fig:ex1-numsol} (generated with
Garrappa's MATLAB implementation of the trapezoidal method \cite{DGU2023,G2018} with a step size $h = 0.1$) 
visualizes that all components tend to zero as $t \to \infty$ and thus confirms this.

\begin{figure}[htb]
	\centering
	\includegraphics[width=0.7\textwidth]{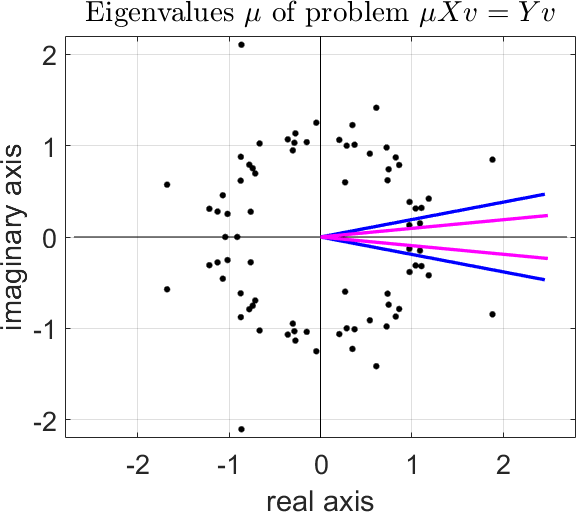}
	\caption{\label{fig:ex1-mu}Location of the 78 finite eigenvalues $\mu$ for the system \eqref{eq:lin-system} with 
		the matrix $A$ given in \eqref{eq:ex1} and $\alpha = \alpha_1$.
		The blue rays separate the eigenvalues in category 3 from those in category~4,
		the magenta rays indicate the boundary between subcategories 4(a) and 4(b).
		It can be seen that subcategory 4(a) is empty, and so the system is asymptotically stable.}
\end{figure}

\begin{figure}[htb]
	\centering
	\includegraphics[width=\textwidth]{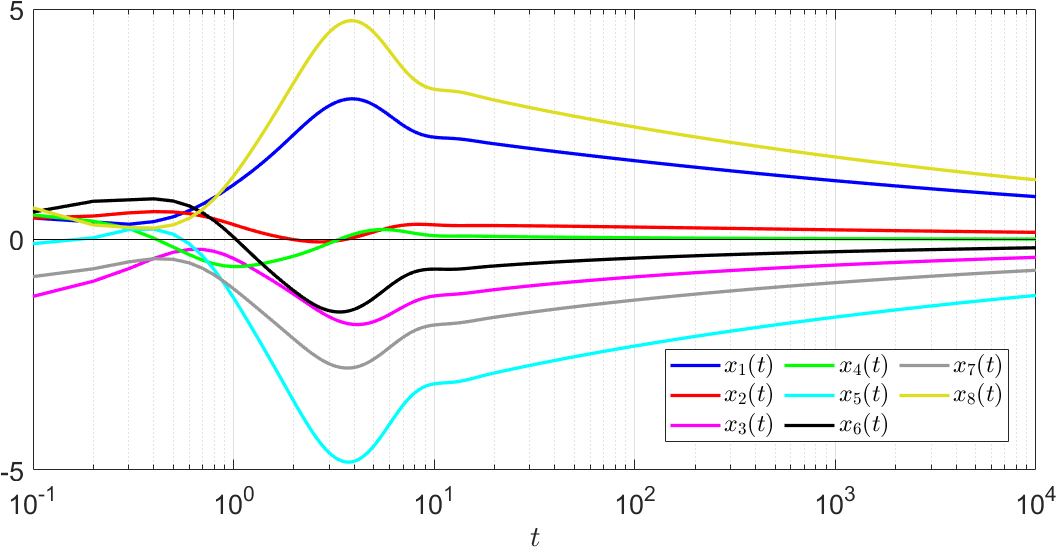}
	\caption{\label{fig:ex1-numsol}The eight components of the numerical solution of the differential equation
		system \eqref{eq:lin-system} with the matrix $A$ given in \eqref{eq:ex1}, $\alpha = \alpha_1$ and initial condition 
		$x(0) = (1, 0, -2, 0.5, -1, 1.5, -2, 0)^{\text T}$. Note that the horizontal axis has a logarithmic scale.}
\end{figure}

An explicit computation of $\chi(\lambda_k)$ for $k = 1, 2, 3, 4$ results in function values satisfying
$| \chi(\lambda_k) | < 7 \cdot 10^{-8}$, which again is at least an indication that the zeros of $\chi$ have been found 
with a good accuracy.

Table \ref{tab:results} summarizes the findings for all four cases. We particularly emphasize the following observations:
\begin{itemize}
\item Since $\det A \ne 0$, we never have any eigenvalues in category 2 according to the classification of 
	Remark \ref{rem:categories}.
\item $N$ is substantially smaller than $\sigma \cdot d$ for all cases that we have considered in this example.
	This underlines the message already mentioned in the derivation of Algorithm \ref{alg:1} 
	in Section \ref{sec:rational} above: Computing only the $N$ finite eigenvalues instead of all 
	$\sigma\cdot d$ (finite and infinite) eigenvalues, i.e.\ using the MATLAB "eigs" function rather than "eig", 
	can significantly reduce the computational cost.
\item When using the asymptotically stable case (a) as a baseline, case (c) is obtained by reducing \emph{all}
	components of $\alpha$. The resulting system is also asymptotically stable. This is expected since it is generally
	known that reducing the orders improves the stability properties \cite{DLL2007,M1998}.
\item In contrast, for cases (b) and (d) some components of $\alpha$ have been increased
	in comparison to case (a). Depending on which 
	components have been increased and by how much, this may or may not lead to the system becoming unstable. 
	A precise analysis of the influence of each component's change on the stability properties is an open
	problem.
\end{itemize}

\begin{table}[htb]
	\centering
	\begin{tabular}{c|c|c|c|c|c|c|c}
		& & & & \multicolumn{3}{c|}{number of eigenvalues $\mu$ of $\mu X v = Yv$} & asymptotic\\
		$\alpha$ & $\sigma$ & $\sigma \cdot d$ & $N = \sum_{k=1}^d q_k$ & 
			in cat.\ 3 & in cat.\ 4(a) & in cat.\ 4(b) & stability \\
		\hline
		$\alpha_1$ & $15$ & $120$ & $78$ & $74$ & $0$ & $4$ & yes  \\
		$\alpha_2$ &  $\phantom{0}8$  & $\phantom{0}64$  & $46$ & $42$  & $0$ & $4$ & yes \\
		$\alpha_3$ & $12$ & $\phantom{0}96$ & $58$ & $54$ & $0$ & $4$ & yes \\
		$\alpha_4$ & $16$ & $128$ & $83$  & $81$ & $2$ & $0$ & no \\
	\end{tabular}
	\caption{\label{tab:results}Summary of the results for the four problems considered associated to the system 
		\eqref{eq:lin-system} with the matrix $A$ from \eqref{eq:ex1} and different choices of $\alpha$.
		The system is unstable if eigenvalues in category 4(a) exist.}
\end{table}

\subsection{A Nonlinear Problem}
\label{subs:nonlin-example}

To demonstrate the application of our approach also to nonlinear problems, we will now look at the
five-dimensional equation system
\begin{equation}
	\label{eq:ex2}
	D^\alpha x(t) = \begin{pmatrix}
				-2.5 x_1(t) + x_2(t) + 0.5 x_5(t) \\
				-x_1(t) - 3 x_2(t) + 1.2 x_3(t) + x_2^2(t) \\
				-1.2 x_2(t) - 2.8 x_3(t) + 1.5 x_4(t) \\
				-1.5 x_3(t) - 3.5 x_4(t) + x_5(t) + x_4^3(t) \\
				-0.5 x_1(t) - x_4(t) - 4 x_5(t)
				\end{pmatrix}
\end{equation}
with $\alpha = (0.577, 0.408, 0.318, 0.367, 0.277)$.
Following the idea outlined in \cite[Subsection 5.2]{DHTT2024}, we split up the system \eqref{eq:ex2} into its linear
and nonlinear parts as
\[
	D^\alpha x(t) = A x(t) + g(x(t))
\]
with
\[
	A = \begin{pmatrix}
			-2.5 & \phantom{-}1.0 & \phantom{-}0.0 & \phantom{-}0.0 & \phantom{-}0.5 \\
			-1.0 & -3.0 & \phantom{-}1.2 & \phantom{-}0.0 & \phantom{-}0.0 \\
			\phantom{-}0.0 & -1.2 & -2.8 & \phantom{-}1.5 & \phantom{-}0.0 \\
			\phantom{-}0.0 & \phantom{-}0.0 & -1.5 & -3.5 & \phantom{-}1.0 \\
			-0.5 & \phantom{-}0.0 & \phantom{-}0.0 & -1.0 & -4.0
		\end{pmatrix}
\]
and 
\[
	g(x) = (0, x_2^2, 0, x_4^3, 0)^{\mathrm T}.
\]
It is clear that $g(0) = 0$ and that the Jacobi matrix $J_g$ of $g$ is given by
\[
	J_g(x) = \begin{pmatrix}
			0 & 0 & 0 & 0 & 0 \\
			0 & \, 2 x_2 \, & 0 & 0 & 0 \\
			0 & 0 & 0 & 0 & 0 \\
			0 & 0 & 0 & \, 3 x_4^2 \, & 0 \\
			0 & 0 & 0 & 0 & 0 
		\end{pmatrix}.
\]
Therefore, $\| J_g(x) \| \to 0$ whenever $x \to 0$ (where $\| \cdot \|$ denotes an arbitrary matrix
norm). By Taylor's Theorem, this implies
\[
	\lim_{r \to 0} \sup \left\{ \frac{\| g(z) - g(\tilde z) \|}{\|z - \tilde z\| }  : \| z \|, \| \tilde z \| \le r \right \}= 0.
\]
From this observation, it follows that the stability properties of the nonlinear system \eqref{eq:ex2} are 
essentially determined by the stability properties of its linear component
\begin{equation}
	\label{eq:ex2lin}
	D^\alpha x(t) = A x(t).
\end{equation}
More precisely, we have:
\begin{itemize}
\item If the linear system \eqref{eq:ex2lin} is unstable then the nonlinear system \eqref{eq:ex2}
	is also unstable \cite[Section 3]{TSSC2017}.
\item If the linear system \eqref{eq:ex2lin} is asymptotically stable then the nonlinear system \eqref{eq:ex2}
	is Mittag-Leffler stable, i.e.\ there exists some $\delta > 0$ such that each solution of the system \eqref{eq:ex2} 
	subject to the initial condition $x(0) = x^*$ with $\| x^* \| \le \delta$
	satisfies $\| x(t) \| = O(t^{-\min_j \alpha_j})$ as $t \to \infty$ and hence
	$\lim_{t \to \infty} x(t) = 0$ \cite[Theorem 5.5]{DHTT2024}.
\end{itemize}
It therefore suffices to investigate the stability properties of the linear system \eqref{eq:ex2lin}, which is what
our algorithm has been designed for. 

To apply our approach to the system \eqref{eq:ex2lin}, we note that we have $d=5$, $\tilde \alpha = \max_j \alpha_j = 0.577$,
$r = (1, 408, 318, 367, 277)$ and $s = (1, 577, 577, 577, 577)$. Our algorithm therefore needs to find the $N = 577 + 408 + 
318 + 367 + 277 = 1947$ finite eigenvalues of the generalized eigenvalue problem $\mu X v = Y v$ whose dimension
is $d \cdot \sigma = 5 \cdot 577 = 2885$. It turns out that all these eigenvalues fall into category 3 of our classification above
(see Figure \ref{fig:ex2-mu}). Therefore, the linear system \eqref{eq:ex2lin} is asymptotically stable, and thus the
nonlinear system \eqref{eq:ex2} is Mittag-Leffler stable in the sense indicated above. To illustrate this observation, 
we have plotted the components
of its solution subject to the initial condition $x(0) = (1, 0, -2, 0.5, -1)^{\mathrm T}$ in Figure \ref{fig:ex2-numsol}.

\begin{figure}[htb]
	\centering
	\includegraphics[height=0.4\textwidth]{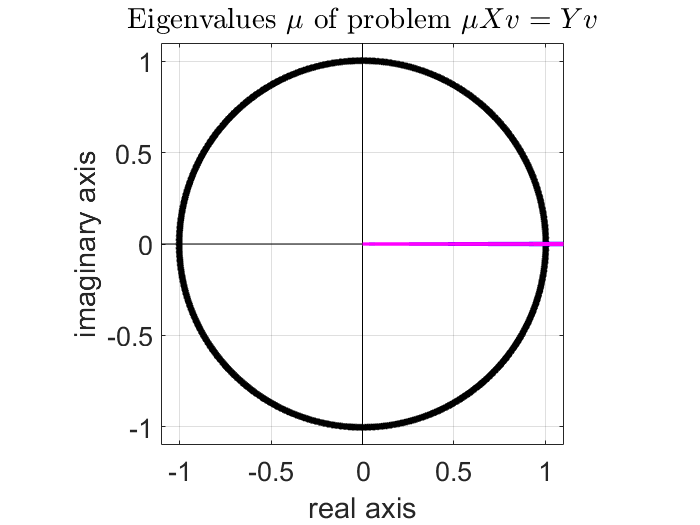}
	\hfill
	\includegraphics[height=0.4\textwidth]{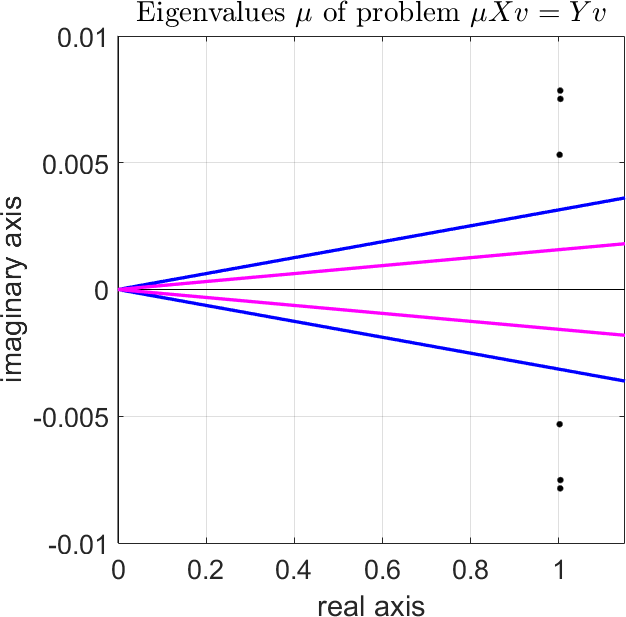}
	\caption{\label{fig:ex2-mu}Location of the 1947 finite eigenvalues $\mu$ for the system \eqref{eq:ex2lin} (left)
		and zoom into this figure to emphasize the regions belonging to subcategories 4(a) and 4(b) (right).
		The blue rays separate the eigenvalues in category 3 from those in category~4,
		the magenta rays indicate the boundary between subcategories 4(a) and 4(b).
		It can be seen that subcategories 4(a) and 4(b) are empty, and so the system is asymptotically stable.}
\end{figure}

\begin{figure}[htb]
	\centering
	\includegraphics[width=\textwidth]{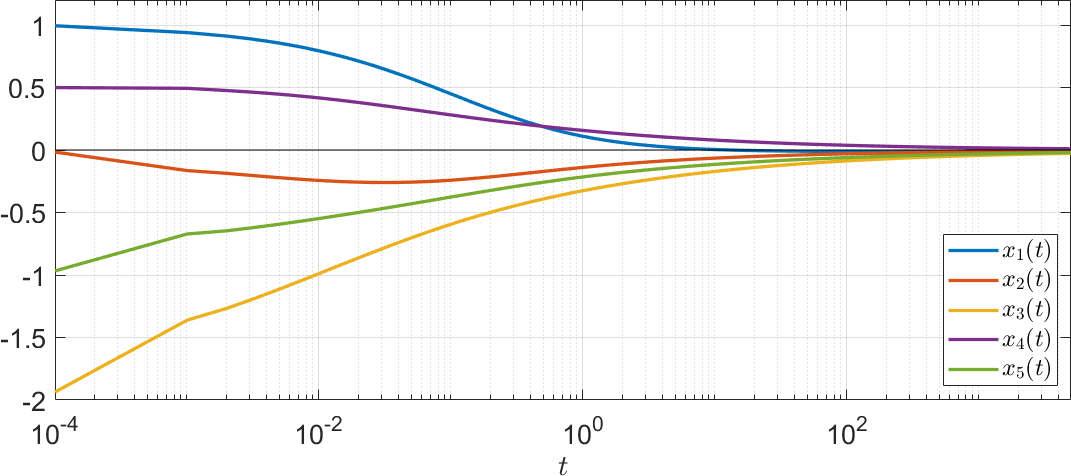}
	\caption{\label{fig:ex2-numsol}The five components of the numerical solution of the differential equation
		system \eqref{eq:ex2} with the initial condition 
		$x(0) = (1, 0, -2, 0.5, -1)^{\text T}$. Note that the horizontal axis has a logarithmic scale.}
\end{figure}

\appendix

\section{Proof of Theorem \ref{thm:46}}
\label{app:proof46}

In this appendix, we present the proof of our Theorem \ref{thm:46}. 
To this end, we first show an auxiliary result. In its formulation, we denote the
characteristic function of the differential equation system $D^\beta x(t) = B x(t)$
with some multi-index $\beta \in (0,1]^d$ and some real $(d\times d)$ matrix $B$ 
by $\chi[B, \beta]$, i.e.\ we set
\[
	\chi[B, \beta](\lambda) = \det \left ( \diag(\lambda^\beta) - B \right)
\]
and use the notation
\begin{align}
	\label{eq:def-delta}
	\delta^2_{\beta}(B) 
	= \inf \{ \| E \|_2: & \text{ } E \in \mathbb C^{d \times d}  \text{ and } \\
	\nonumber
	& \text{ at least one zero of } \chi[B+E, \beta] 
			\text{ is not in the open left half of } \mathbb C \}.
\end{align}

\begin{lemma}
	\label{lem:ext46}
	Let $d \in \mathbb N$, $A$ be a given real $(d \times d)$ matrix and $\beta \in (0,1]^d$. 
	If there exists some $\epsilon > 0$ such that $\delta^2_{\beta}(A) \ge \epsilon$ then all zeros
	of $\chi[A, \beta]$ are located in the open left half of the complex plane.
\end{lemma}

\begin{proof}
	Assume that the statement is false, i.e.\ that $\chi[A, \beta]$ has a zero which is not 
	in the open left half of $\mathbb C$.
	Then, denoting the null matrix of size $d \times d$ by $0$, we see that
	$\chi[A+0, \beta]$ has a zero outside of the open left half of $\mathbb C$.
	Therefore, by definition of $\delta^2_\beta$, see eq.~\eqref{eq:def-delta},
	we have that $\delta^2_{\beta}(A) = 0$ in contradiction with our assumption that
	$\delta^2_{\beta}(A) \ge \epsilon > 0$.
	\qed
\end{proof}

\begin{proof}[of Theorem \ref{thm:46}]
	We consider the following three statements:
	\begin{enumerate}
	\item[(a)] The system \eqref{eq:lin-system} is asymptotically stable.
	\item[(b)] There exist some $\epsilon > 0$ and an $\epsilon$-rational approximation $\alpha^*$ of $\alpha$ 
		such that the corresponding system \eqref{eq:lin-system2} is asymptotically stable and 
		$\delta^2_{\alpha^*}(A) \ge \epsilon$.
	\item[(c)] There exist some $\epsilon > 0$ and an $\epsilon$-rational approximation $\alpha^*$ of $\alpha$ 
		such that $\delta^2_{\alpha^*}(A) \ge \epsilon$.
	\end{enumerate}
	By  \cite[Theorem 4.6]{DHTT2024}, statement (a) is equivalent to (b). 
	It therefore remains to show that (b) and (c) are also equivalent to each other. 
	The implication (b) $\Rightarrow$ (c) is trivial. For the implication
	(c) $\Rightarrow$ (b) we can see that, by Lemma \ref{lem:ext46}, (c) implies that all zeros of
	$\chi[A, \alpha^*]$ are in the open left half of $\mathbb C$, and this implies the
	asymptotic stability of the system \eqref{eq:lin-system2}.
	The statement of Theorem \ref{thm:46} then follows since
	$\delta^2_{\alpha^*}(A) = \min \{ \nu(\diag(z^{\alpha^*}) - A) : \re z = 0 \}$,
	see \cite[eq.~(29)]{DHTT2024}.
	\qed
\end{proof}

\color{black}

\section*{Declarations}

\paragraph{Funding}
This work has not received any external funding. 

\paragraph{Data Availability}
The software code created within the study described in this paper is available for download 
at \cite{DH:zenodo-stability}. No other datasets were generated.

\paragraph{Conflicts of Interest}
The authors declare that there is no conflict of interest.

\paragraph{Author Contributions}
Both authors contributed to the development and analysis of the algorithm
and to the writing and reviewing of the manuscript.

\end{document}